\documentclass{amsart}
\usepackage{verbatim,amsfonts,color,mathrsfs,stmaryrd,graphicx}

\theoremstyle{plain}
\newtheorem{Thm}{Theorem}

\newtheorem{Lem}[Thm]{Lemma}
\newtheorem{Prop}[Thm]{Proposition}

\theoremstyle{definition}
\newtheorem{Def}[Thm]{Definition}
\newtheorem{Rmk}[Thm]{Remark}
\newtheorem{Eg}[Thm]{Example}

\theoremstyle{remark}

\begin{document}


\title{Infinitely many lattice surfaces with special pseudo-Anosov maps} 
\author{Kariane Calta}
\address{Vassar College} 
\email{kacalta@vassar.edu }

\author{Thomas A. Schmidt}
\address{Oregon State University\\Corvallis, OR 97331}
\email{toms@math.orst.edu}
\date{3 October 2012}


\begin{abstract}  We give explicit pseudo-Anosov homeomorphisms with vanishing Sah-Arnoux-Fathi invariant.    Any translation surface whose Veech group is commensurable to any of a large class of triangle groups is shown to have an affine pseudo-Anosov homeomorphism of this type.    We also apply a reduction to finite triangle groups and thereby show the existence of non-parabolic elements in the periodic field of certain translation surfaces.  
\end{abstract}

\maketitle


\section{Introduction} 
\subsection{Motivation and History} 

Thurston introduced the notion of  pseudo-Anosov homeomorphism in his proof of the Nielsen-Thurston classification theorem of surface diffeomorphisms, eventually published in \cite{T}.       Also in that work,  Thurston exhibited {\em affine pseudo-Anosov} homeomorphisms;  that is, he gave a flat structure on a surface with appropriate conical singularities,  for which the map in question is in the affine diffeomorphism group:  it is locally affine,  with constant linear part.   An affine pseudo-Anosov's linear part in the flat coordinates is given by a real hyperbolic matrix, and the dilatation is simply the  eigenvalue of largest absolute value for this matrix; the stable directions of the affine pseudo-Anosov map have directions corresponding to the fixed points of its linear part.   The surface with flat structure involved is what is now is called a {\em translation surface}, see below for a definition.  Veech \cite{V} showed that the group of linear parts of the affine diffeomorphisms of a translation surface forms a Fuchsian group;  this group is now often called the {\em Veech group} of the translation surface.   A {\em lattice surface} is a translation surface for which the Veech group is a lattice in $\text{PSL}(2, \mathbb R)$; that is, it is of cofinite volume with respect to Haar measure.   The celebrated Veech dichotomy states in particular that lattice surfaces have what McMullen has dubbed ``optimal dynamics''.   
 
 Any full transversal to the linear flow in a fixed direction a translation surface defines,  by first return, an interval exchange transformation.    The Sah-Arnoux-Fathi (SAF) invariant of the interval exchange transformation gives what McMullen reformulates as his ``flux''; he shows in \cite{Mc} that the vanishing of this flux for an affine pseudo-Anosov homemomorphism implies zero average drift for the leaves of the foliation in its expanding direction; see also \cite{ABB}, \cite{LPV}, \cite{LPV2}.
 
     The first examples of an affine pseudo-Anosov with vanishing SAF invariant was given by Arnoux and Yoccoz \cite{AY}; they used a construction involving suspension of interval exchange transformations.       Arnoux and Schmidt \cite{AS} found special pseudo-Anosov maps of the lattice surfaces given by gluing two copies of the regular $n$-gon together along opposite edges for $n \in \lbrace 7, 9, 14, 18, 20, 24 \rbrace$.  The discovery of these maps was especially surprising considering that the surfaces themselves are among the well-studied first examples of Veech of non-arithmetic lattice surfaces, \cite{V}.   Recently,  McMullen \cite{Mc} has communicated an  example found by Lanneau of a special pseudo-Anosov homeomorphism with vanishing SAF invariant on a genus three surface.

 We give a large class of  pseudo-Anosov homeomorphisms with vanishing SAF invariant.   We do this by finding so-called ``special'' affine pseudo-Anosov homeomorphism.  
 Long and Ried \cite{LR} call a hyperbolic element of a Fuchsian group {\em special} if its eigenvalues lie in the trace field of the group.   Accordingly,  an affine pseudo-Anosov homeomorphism is called special if the eigenvalues of its linear part lie in the trace field of the Veech group of its translation surface.   Calta and Smillie \cite{CS} show that under fairly mild hypotheses, see Lemma~\ref{l:special},  one can normalize a translation surface so that the set of (cotangents of) directions for flows with vanishing Sah-Arnoux-Fathi invariant forms the trace field (and infinity).    With this normalization, each  special  affine pseudo-Anosov homeomorphisms has vanishing Sah-Arnoux-Fathi invariant.   
 
\subsection{Main Result} 
We prove the following result, where $\Delta(m,n,\infty)$ is the usual hyperbolic triangle group;  that  is,  the orientation preserving subgroup of the group generated by reflections in the sides of a hyperbolic triangle having one ideal vertex and the other two having angles $\pi/m$ and $\pi/n$.   (See below for the definition of the invariant trace field of a Fuchsian group.)

\begin{Thm}\label{t:nonobstrSpecPA}  Suppose that both $m,n$ are even,  and there is equality of the trace field and invariant trace field for the triangle group $\Delta(m,n, \infty)$.   Then any translation surface whose Veech group is commensurable to $\Delta(m,n, \infty)$ has special pseudo-Anosov elements in its affine group. 
\end{Thm} 

Note that Bouw-M\"oller \cite{BM}, as confirmed by Hooper \cite{H},  have shown that for each signature, translation surfaces satisfying these hypotheses do exist.    Using Hooper's presentation of the Bouw-M\"oller surfaces,  we give fully explicit special affine pseudo-Anosov homeomorphisms,  see Proposition~\ref{p:explicit}, Example~\ref{eg:HooperSurfs} and Figure~\ref{eight4Fig}.\\

We also show that certain other signatures of triangle groups are such that any translation surface whose Veech group is commensurable to a group of the signature must have infinitely many non-periodic directions with vanishing SAF invariant.    Informed by this, we obtain a further new special affine pseudo-Anosov homeomorphism,  see Example~\ref{eg:sevenSeven} and Figure~\ref{f:fatSeven7andFluxMapped}.  This particular example is on a surface whose the trace field is a cubic field,  and thus the flow in the expanding direction of the pseudo-Anosov map has rank three in McMullen's sense; we thus can give a pictorial representation of this flow on a genus 15 surface of the type that McMullen \cite{Mc} gives for each of the cubic example of Arnoux and Yoccoz and the recent cubic example of Lanneau, again see Figure~\ref{f:fatSeven7andFluxMapped}. \\

We note here that we had previously \cite{CaltaSchmidt} built continued fraction algorithms for the Veech groups of the translation surfaces studied by \cite{W} --- these groups are Fuchsian triangle groups, of  signature $(3,n, \infty)$---,  in anticipation of using them to find non-parabolic directions on the surfaces or even special pseudo-Anosov homeomorphisms. The Ward groups have proven to be significantly more resistant to the search for such phenomena than were the original Veech examples, and of course than the translation surfaces with Veech group of signatures that we treat here.  

\subsection{Thanks} Parts of this work arose out of our conversations during the 2011 Oberwolfach workshop ``Billiards, Flat Surfaces, and Dynamics on Moduli Spaces'',   we warmly thank the organizers for that opportunity.  

\bigskip 
\section{Background} 

\subsection{Translation surfaces and Veech groups} 

A translation surface is a  $2$-manifold with finitely many marked points and an atlas whose transition functions are translations. This is equivalent to the definition of a translation surface as a disjoint union of finitely many polygons $P_1, \cdots, P_n$ in $\mathbb R^2$ glued along parallel edges to form a closed surface.  The marked points are cone points, which can arise at vertices of the $P_i$ when  too many polygons are glued around a single vertex, resulting in a total angle at that vertex of $2k\pi$ where $k \ge 1$ is an integer. Equivalently,  a translation surface can be seen as a pair $(M,\omega)$ where $M$ is a Riemann surface and $\omega$ an abelian differential on $M$ --- away from the zeros of the abelian differential,  integration of the abelian differential gives local coordinates with transition functions that are translations.

The group $\text{SL}(2,\mathbb R)$ acts on the moduli space of translation surfaces, preserving genus and the number and order of cone points. In the polygonal model, if $S=P_1, \dots, P_n$  is a surface and $g \in \text{SL}(2,\mathbb R)$ then $gS$  is defined as $gP_1 \cup  \dots  \cup gP_n$. The stabilizer of a surface $S$ under this action is called its {\it Veech group}, which is always a non-cocompact Fuchsian group.  Generically this group is empty but occasionally it is a {\it lattice} subgroup of $\text{SL}(2,\mathbb R)$, which is to say that it is has finite covolume. In this case, we refer to the surface as a {\it lattice surface}.  Veech \cite{V} proved that a pair of regular $n$-gons glued together along parallel sides forms a lattice surface for  $n \geq 4$.    Whether or not a surface is a lattice surface has profound implications for the dynamics of the linear flow on the surface.  Veech \cite{V} proved that in any direction $v$ on a lattice surface, the orbits of the linear flow in that direction are either closed or connect two cone points, or all orbits in that direction are uniformly distributed on the surface.  In the literature, this dynamic dichotomy has come to be called ``optimal dynamics".

\subsection{The Sah-Arnoux-Fathi Invariant}

Suppose that $f$ is an interval exchange transformation (iet) on a finite interval $I$, that is, a piecewise linear orientation preserving isometry of $I$.  Then by definition $f$ exchanges $n$ intervals $I_i$ of lengths $l_i$ for $i=1,...,n$ by translating each $I_i$ by the amount $t_i$.  One can associate to $f$ a certain invariant known as the Sah-Arnoux-Fathi (SAF) invariant that takes values in $\mathbb R \wedge_{\mathbb Q} \mathbb R$ and is defined as $\sum_{i=1}^n l_i \wedge_{\mathbb Q}  t_i$.  The SAF invariant is a central tool in the study of the dynamics of the linear flow on translation surfaces.  Given a direction $v$ on a surface, one can choose an interval transverse to the orbits of the  flow in the direction $v$ that meets every orbit.  The first return map to this interval is an iet.  If the flow in the direction $v$ is periodic then the associated SAF invariant is zero. The converse however is false. The dynamics of the flow in SAF zero directions has been an object of recent interest and in this paper we show that there are surfaces for which the flow in a particular SAF zero direction is the expanding direction of a pseudo-Anosov homeomorphism.  

There is another way to define the SAF invariant of a direction on a translation surface using the $J$ invariant of Kenyon and Smillie \cite{KS}.  The $J$  invariant of a polygon takes values in $\mathbb R^2 \wedge_{\mathbb Q} {\mathbb R^2}$.  If the vertices of a polygon $P$ are $v_0,\dots,v_n$, then $J(P)=\sum_{i=0}^n v_i \wedge v_{i+1}$ where $v_{n+1}=v_0$. Since a translation surface $S$ can be realized as a disjoint union of polygons $P_i$ for $i=1, \dots, n$ glued along parallel sides, we define $J(S)$ to be $\sum_{i=1}^n J(P_i)$. Then the projection $J_v(S)$ of $J(S)$ in the direction $v$ is the SAF invariant of the iet which is the first return map on a full transversal to the linear flow in the direction $v$.  It is not hard to see that the SAF invariant of a periodic iet is zero.  
Thus the SAF invariant in a parabolic direction on a surface is zero.  

Also note that Lemma 2.4 of the  Appendix of Calta's \cite{C} directly implies that the SAF invariant in the form of $J_v(S)$ is constant on $\text{SL}_2(\mathbb R)$-orbits, in the sense that  $J_v(S) = J_{Av}(A\circ S)$ for any $A\in \text{SL}_2(\mathbb R)$.

\subsection{Fields and translation surfaces} 

\indent If a translation surface has at least three directions of vanishing SAF invariant, then Calta and Smillie  \cite{CS} show that the surface can be normalized by way of the $\text{SL}_2(\mathbb R)$-action so that the directions with slope $0$, $1$ and $\infty$ have vanishing SAF invariant and prove that on the normalized surface the set of slopes of directions  with vanishing invariant is a field union with infinity.  A translation surface so normalized is said to be in {\it standard form}, and the field so described is called the {\it periodic direction field}. In this paper, we are primarily interested in directions on a surface that come from the periodic direction field. 

 On the other hand, Kenyon and Smillie \cite{KS} defined the {\it holonomy field} of a translation surface as 
the smallest field over which the set of holonomy vectors is contained in a two dimensional vector space.    A holonomy vector  is associated via the developing map to a closed, nonsingular curve on the surface or to a closed curve that is a union of saddle connections.  
 
Gutkin and Judge define the {\it trace field} of a surface to be the extension of $\mathbb Q$  generated by the traces of the elements of its Veech group.  Since the trace is a conjugacy invariant, the trace field of a given surface is the same for as that of any other surface in its $\text{SL}(2,\mathbb R)$  orbit.   Calta and Smillie \cite{CS} show that if $S$ is a lattice surface, then the holonomy, trace and periodic direction fields are all equal. 
 
\bigskip

 The following is a direct implication of the Calta-Smillie \cite{CS} result that the periodic field of a translation surface in standard form equals its trace field.  
\begin{Lem}\label{l:special}    On a translation surface in standard form, the stable directions of an affine pseudo-Anosov have vanishing Sah-Arnoux-Fathi   invariant  if and only if the pseudo-Anosov is special.
\end{Lem} 

\subsection{  Triangle groups and realizability}

Of central importance to us is the fact that triangle groups are (up to finite index) realized as Veech groups.  

\begin{Thm}\label{t:bouwMH}[Bouw-M\"oller, Hooper]    Every hyperbolic triangle group with parabolic elements is commensurable to a group realized as the Veech group of a translation surface.  
\end{Thm} 

However, not every full triangle group can be realized as a Veech group.    Hubert and Schmidt \cite{HS} remarked that one can use the fundamental observation of Kenyon and Smillie \cite{KS} that the trace field of a translation surface is generated by the trace of any of its affine pseudo-Anosov homeomorphisms to show that no triangle group of signature $(2, 2n, \infty)$ can be realized as a Veech group.    Hooper \cite{H} uses the observation in the form that a Fuchsian can only be realized as a Veech group if its trace field equals the field generated over the rationals by the traces of the squares of elements of the group;  this latter field is called the {\em invariant trace field} of the group,  as Margulis proved that it is an invariant of the (wide) commensurability class of the group, see \cite{MR} for a discussion.

\section{Special affine pseudo-Anosov homeomorphisms}   In this section, we focus on the arithmetic of triangle groups in order to find the linear parts of special affine pseudo-Anosov homeomorphisms.   To do this,  Lemma~\ref{l:specHypMatrix} is key.     It allows us to prove a more precisely worded version of our main theorem, Theorem~\ref{t:mainReworded},   thus giving fully explicit pseudo-Anosov maps, as shown by Example~\ref{eg:HooperSurfs}  and Figure~\ref{eight4Fig}.     In the final subsection,  we give some results about the groups for which the full triangle group is never a Veech group.
\bigskip 
 
 \subsection{A special hyperbolic matrix}
Our approach here is centered on properties of groups realized as Veech groups of translation surfaces.   Since any parabolic direction on a translation surface has flow of vanishing Sah-Arnoux-Fathi invariant we use the following variation of a term used by Calta and Smillie \cite{CS} .

\begin{Def}    A Fuchsian group  is  in (group) {\em parabolic standard form} if its set of parabolic fixed points (for its action on the Poincar\'e upper half-plane)  includes $0$, $1$ and $\infty$.   
\end{Def}

Note that since the Sah-Arnoux-Fathi invariant of a periodic flow is zero,  whenever the Veech group of a translation surface is in parabolic standard form,  the surface itself is in the standard form defined by Calta and Smillie.
\bigskip

To expedite discussion,  we take a specific representation for each of the triangle groups we consider. Let  $G_{m,n}$ generated by  

\begin{equation}\label{e:generators} 
A = \begin{pmatrix} 1& 2 \cos \pi/m + 2 \cos\pi/n\\
                                      0&1\end{pmatrix},\, B = \begin{pmatrix} 2 \cos \pi/m& 1\\
                                      -1&0\end{pmatrix},\,   C = \begin{pmatrix} -2 \cos \pi/n& 1\\
                                      -1&0\end{pmatrix}\,,
\end{equation}
and note that $C= AB$.  The group is easily verified to be a Fuchsian triangle group of signature $(m,n,\infty)$.  The trace field of $G_{m,n}$ is $K_{m,n} = \mathbb Q(\cos \pi/m, \cos \pi/n)$, see   p. 159 of  \cite{MR}. 
 \bigskip 

\begin{Lem}\label{l:specHypMatrix}  Suppose that $m,n$ are distinct and even. If $\alpha$ is a nonzero finite parabolic fixed point of $G_{m,n}$, then multiplication by $\alpha^{-1}$ defines a transformation that conjugates $G_{m,n}$ to a group in parabolic standard form, with special hyperbolic elements.  
\end{Lem} 

\begin{proof} Being generated by $A$ and  $B$, the group $G_{m,n}$ clearly has  infinity as a parabolic  fixed point.  Since $B$ sends $0$ to infinity, $0$ is also a parabolic  fixed point.  

Recall that the product of any two distinct elliptic elements of order two  in $\text{PSL}(2, \mathbb R)$ is a hyperbolic element 
Here, we certainly have that $ B^{m/2}, C^{n/2}$ are elliptic of order two, their product is thus hyperbolic.    Now,   for each integer $k$, 
\begin{equation}\label{e:powersOfBandC} B^k = \begin{pmatrix} b_{k+1}  & b_k\\
                                      -b_k&-b_{k-1}\end{pmatrix}\,, \;\; C^k = (-1)^k\begin{pmatrix} c_{k+1}  & -c_k\\
                                      c_k&-c_{k-1}\end{pmatrix}\,
\end{equation}
where $b_k = \sin \frac{k \pi}{m}/\sin \frac{\pi}{m}$, and $c_k = \sin \frac{k \pi}{n}/\sin \frac{\pi}{n}$, see say \cite{BKS}; thus,  one finds that
\[ B^{m/2}\cdot (\pm 1) = \mp 1, \;\; \; C^{n/2}\cdot (\pm 1) = \mp 1\,.\]
Thus,  their product $B^{m/2} C^{n/2}$ fixes both $-1, 1$. 

  Since $G_{m,n} \subset \text{SL}_2(\mathcal O_K)$,  all finite parabolic fixed points of $G_{m,n}$ lie in the field $K_{m,n}$.  Let $\alpha$  be as in our hypotheses. By the triple transitivity of $\text{SL}_2(\mathbb R)$ acting on the real projective line,   there is an element $M$ that sends $\alpha$ to $1$ while fixing each of zero and infinity.   But, elementary considerations of $2 \times 2$-matrices show that the action of $M$ is simply multiplication by $\alpha^{-1}$.     The conjugation of $G_{m,n}$ by $M$ is clearly in parabolic standard form and   the hyperbolic fixed point $1$ of $G_{m,n}$ corresponds to the point $\alpha^{-1}$ --- this is an element of $K$ fixed by some hyperbolic element of the conjugate group.   That is, the conjugate group has special hyperbolic elements.   
\end{proof} 

\subsection{Special affine pseudo-Anosov homeomorphisms}

Recall that W.~P.~Hooper determined conditions such that the  triangle group $\Delta(m,n, \infty)$ has its trace field  equal to its invariant trace field;  the conditions are given  in terms of the indices $m$, $n$ and their greatest common divisor.   Inequality of the two fields is an obstruction to the group being a Veech group of any translation surface.  
 
\begin{Def}   Given a pair of integers $m,n$, let $\gamma = \gcd(m,n)$.   We say that the pair $m, n$  is {\em unobstructed} if neither of the following conditions hold: 
\begin{enumerate} 
\item $\gamma = 2$; 
\item $m/\gamma$ and $n/\gamma$ are both odd.
\end{enumerate} 
\end{Def}


Our main result can be more precisely stated as follows.

\begin{Thm}\label{t:mainReworded}      Let $m,n$ be an unobstructed pair of even integers,  and suppose that $\mathcal S$ is a  translation surface whose Veech group is commensurable to $G_{m,n}$.   Then    some power of $B^{m/2}C^{n/2}$ is the derivative of a special affine pseudo-Anosov automorphism of $\mathcal S$.
 \end{Thm} 
\begin{proof}   Suppose that $\mathcal S$ is a translation surface whose Veech group is commensurable to $G_{m,n}$.   Using the $\text{SL}_2(\mathbb R)$-action,  we may assume that $\mathcal S$  has as its Veech group a finite index subgroup of $G_{m,n}$.        Lemma~\ref{l:specHypMatrix} now provides an element of $\text{SL}_2(\mathbb R)$ conjugating $G_{m,n}$ into parabolic standard form while conjugating $B^{m/2}C^{n/2}$ to  a special hyperbolic matrix.  

The set of parabolic fixed points is unaltered by passage to a  finite index subgroup,   thus the image of $\mathcal S$ by this conjugating element is  in standard form.    By the work of Calta and Smillie \cite{CS} the set of (non-vertical) directions for which the flow has vanishing Sah-Arnoux-Fathi invariant forms a field, the periodic field,  and since the translation surface certainly has some affine pseudo-Anosov automorphism,  this equals the trace field of the Veech group.

    Now,  there is some nonzero power of our (special) hyperbolic element of the larger group that belongs to the finite index subgroup.    Since the two fixed points are common to the cyclic subgroup generated by a hyperbolic element, this element of the Veech group fixes points in the periodic field.   From this,  the corresponding pseudo-Anosov map has vanishing Sah-Arnoux-Fathi invariant.     Finally,  by  Lemma~2.4 of the Appendix in \cite{C},   any $M \in \text{SL}_2(\mathbb R)$ sends this direction to a direction on the $M$-image of this surface that also has vanishing  Sah-Arnoux-Fathi invariant.   In particular,  we can return in this way to the original direction and surface we began with;  the result thus holds.  
\end{proof} 

\bigskip 
Our construction easily leads to explicit examples.   
 Hooper \cite{H} explicitly realizes the Bouw-M\"oller translation surfaces in two different ways.   First, by way of grid graphs presenting the combinatorics of the
intersections of horizontal and vertical cylinders so as to apply
the Thurston construction \cite{T}; the resulting surfaces Hooper denotes   $(X_{m,n}, \omega_{m,n})$.   Second,   as translation surfaces  he denotes as $(Y_{m,n}, \eta_{m,n})$,  formed by appropriately identifying sides of semi-regular polygons.   In the case where both $m$ and $n$ are even,  there is a natural involution on the surface.   Hooper denotes the resulting respective quotients  as $(X_{m,n}^{e}, \omega_{m,n}^{e})$, $(Y_{m,n}^{e}, \eta_{m,n}^{e})$.   
Hooper shows  that the Veech group of $(X_{m,n}, \omega_{m,n})$ is an index two subgroup of $G_{m,n}$ and that the transformation $z \mapsto D_{\mu}(z) =  (\csc \pi/n) z - \cot \pi/n$ conjugates this to the Veech group of    $(Y_{m,n}, \eta_{m,n})$.   

\begin{Prop}\label{p:explicit}     Suppose $m,n$ is an obstructed pair of even integers.   Then on each of  Hooper's translation surfaces
 $(Y_{m,n}, \eta_{m,n})$ and $(Y_{m,n}^{e}, \eta_{m,n}^{e})$,  the flow in the direction  $(1- \cos \pi/n)/(\sin \pi/n)$ is  a stable direction of a special affine Anosov homeomorphism.     Furthermore,  letting 
\[ \lambda = \dfrac{\cos \frac{\pi}{m} \cos \frac{\pi}{n} + \cos \frac{\pi}{m}  + \cos \frac{\pi}{n}  + 1}{\sin \frac{\pi}{m} \sin \frac{\pi}{n} }\,, 
\] 
the dilatation of this Anosov homeomorphism is $\lambda$ if four divides $\gcd(m,n)$ and $\lambda^2$ otherwise.  
\end{Prop} 
\begin{proof}   Since any power of of $B^{m/2}C^{n/2}$ fixes the direction $z=1$,  we find that the direction   $D_{\mu}(1) = (1- \cos \pi/n)/(\sin \pi/n)$ determines a  flow on each of  $(Y_{m,n}, \eta_{m,n})$ and  $(Y_{m,n}^{e}, \eta_{m,n}^{e})$ with vanishing Sah-Arnoux-Fathi invariant.     Hooper shows that  $B^2, C^2$ (in our notation) are in the Veech group of $(X_{m,n}, \omega_{m,n})$, whereas neither $B$ nor $C$ is.  It follows that $B^{m/2}C^{n/2}$ itself is in this group if and only if four divides $\gcd(m,n)$.   Otherwise it is the square of this special hyperbolic that belongs to the group.

    The dilatation of an affine  pseudo-Anosov automorphism is the larger of the two eigenvalues of its linear part, and here this is the same as that of $B^{m/2}C^{n/2}$ or of its square.  Thus, the dilatation is  as claimed.
\end{proof}

\begin{Eg}\label{eg:HooperSurfs}   
  In Figure~\ref{eight4Fig},  we show the result when $(m,n) = (8,4)$;  the translation surface  $(Y_{8,4}^{e}, \eta_{8,4}^{e})$ is a suspension surface over an interval exchange transformation on eleven intervals with permutation in the usual redundant notation 
   \setcounter{MaxMatrixCols}{20}
  \[  
   \begin{pmatrix} 1&2&3& 4&5&6&7&8&9&10&11\\
                              7&5&2&10&3&1&11&8 &4&6  &9\end{pmatrix}\,.
  \]    The dilatation of this pseudo-Anosov map is the quartic number  $3 + 2 \sqrt{2} + \sqrt{20 + 14 \sqrt{2}}$.

\begin{figure}[h]
\includegraphics{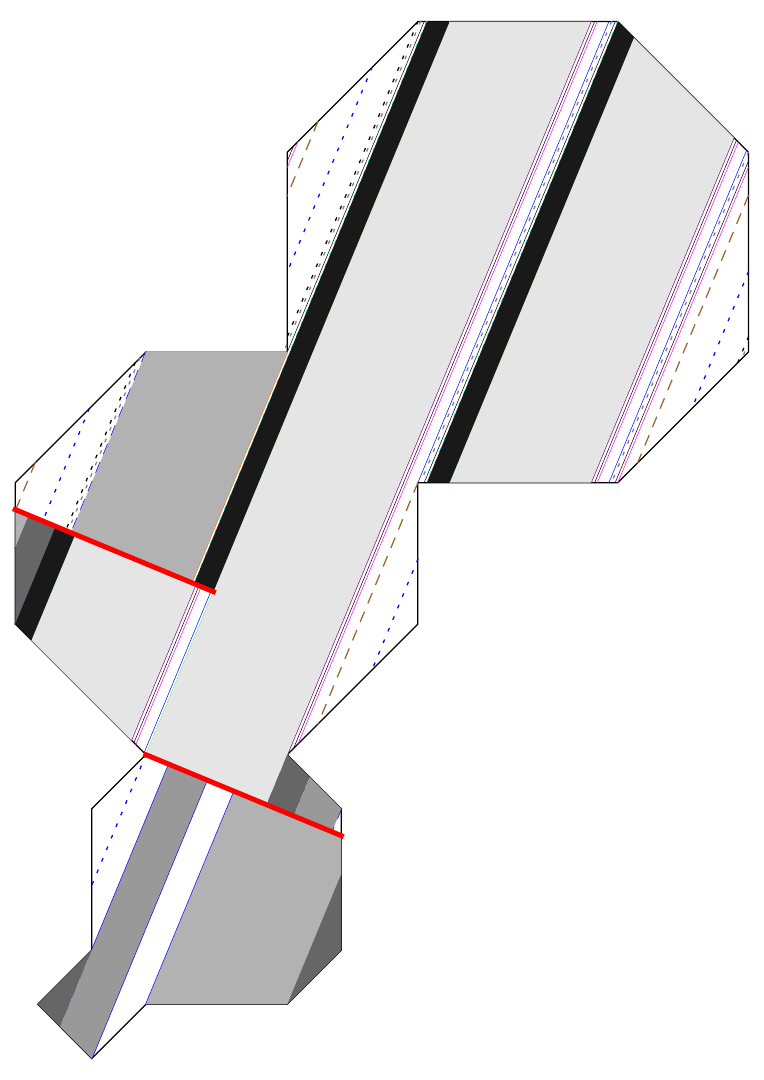}
\caption{A pseudo-Anosov map with vanishing Sah-Arnoux-Fathi invariant, indicated as zippered rectangles on Hooper's translation surface $(Y_{8,4}^{e}, \eta_{8,4}^{e})$.  Thick (red) intervals comprise the transversal to the flow; rectangles 1,2,3,9,11 are shaded.}
\label{eight4Fig}
\end{figure}
\end{Eg}

\subsection{Remarks on the ``obstructed'' setting} 

We now show that Lemma~\ref{l:specHypMatrix} cannot provide information about the existence of special pseudo-Anosov maps when the group $G_{m,n}$ has its trace field unequal to its invariant trace field.  
\begin{Prop}\label{p:parabFxPtsNotInInvarTraceField}     If $m, n$ are such that the trace field of $G_{m,n}$ does not equal its invariant trace field,  then no nonzero parabolic fixed point of $G_{m,n}$ lies in its invariant trace field.
\end{Prop} 
\begin{proof}  Recall that the trace field $K_{m,n}$ of $G_{m,n}$ is generated over $\mathbb Q$ by the pair $\{\cos \pi/m, \cos \pi/n\}$ and Hooper shows that the triple $\{\cos 2\pi/m, \cos 2\pi/n,  \cos \pi/m \cos \pi/n\}$ generates
the invariant trace field, $k_{m,n}$, see also p. 159 of \cite{MR}.    Define $\delta = 2 (\cos 2\pi/n + 1)$; note that since   $n >2$ our $\delta$ is positive.  One easily shows the equality  $ K_{m,n} =  k_{m,n}(\sqrt{\delta})$ since first a standard double angle formula implies that $2 \cos  \pi/n = \sqrt{\delta}$,   and then of course $\cos \pi/m = (\cos \pi/m \cos \pi/n)/\cos \pi/n$.

Any pair of the elements in ~\eqref{e:generators} generate $G_{m,n}$, thus this group is generated by 
\[C = \begin{pmatrix} -\sqrt{\delta}& 1\\
                                      -1&0\end{pmatrix},\; A = \begin{pmatrix} 1&b \sqrt{\delta}\\
                                      0&1\end{pmatrix},\, 
\]
with 
\[b = \dfrac{\delta + 4 \cos \frac{\pi}{m} \cos \frac{\pi}{n}}{\delta}\,.\] 
Of course,  $b \in  k_{m,n}$ and thus $b \sqrt{\delta} \in   K_{m,n} \setminus k_{m,n}$.  We thus define two types of elements of $G_{m,n}$ --- an {\em even} element is of the form  
\[ \begin{pmatrix} a&b \sqrt{\delta}\\
                                      c \sqrt{\delta}&d\end{pmatrix}\,,\]
with $a,b,c,d \in     K_{m,n}$; similarly,  an {\em odd} element is of the form                                   
\[ \begin{pmatrix} a\sqrt{\delta}&b \\
                                      c &d\sqrt{\delta}\end{pmatrix}\,;
\]
in particular, $A$ is of even type while $C$ is of odd type.   
Multiplication thus is similar to addition of integers in that the product of two elements of the same type is an even element, whereas the product of two elements of distinct types is an odd element.   Now, any element of $G_{m,n}$ is a product in powers of $A$ and $C$,  and therefore is of one of the two types.   But,  the image of $\infty$ under such a group element is then of the form $a/c\sqrt{\delta}$ or $a\sqrt{\delta}/c$; any nonzero $a/c \in k_{m,n}$ of course has a multiplicative inverse in this field, and hence were the image of infinity to be in $k_{m,n}$, the contradiction of $\sqrt{\delta}$ also belonging to this field would follow.
\end{proof} 

The notion of even and odd subgroups of $G_{m,n}$ defined in the proof of Proposition~\ref{p:parabFxPtsNotInInvarTraceField}  generalizes a notion for the  Hecke triangle groups,  of signature $(2,q, \infty)$ with $q$ even; see say Rankin's {\em review} in the AMS Mathematical Reviews: MR0529968 (80j:10037).    

\begin{Lem} \label{l:evenGpOnly}    Suppose that  $m, n$ are such that the trace field of $G_{m,n}$ does not equal its invariant trace field.   If $H$ is a subgroup of $G_{m,n}$ that contains a parabolic element and $H$ is realized as a Veech group of some translation surface, then $H$ is contained in the even  subgroup of $G_{m,n}$. 
\end{Lem}
 
 \begin{proof} The Kenyon-Smillie result shows that if a hyperbolic element of $G_{m,n}$ lies in some Veech group, then this element is contained in the even subgroup.    Suppose now that some subgroup $H$ is {\em not} contained in the even subgroup of $G_{m,n}$.   Since parabolic elements of $G_{m,n}$  are conjugates of powers of $A$ they are even elements;   thus $H$ either contains a hyperbolic element and we are done, or it contains an odd elliptic element.   Denote this odd elliptic element by $E$.   Choose a parabolic element $P\in H$, so $P = M A^r M^{-1}$ for some $M \in G_{m,n}$ and some integer $r$.  Let $F = M^{-1} E M$, so that 
 \[PE = M A^r M^{-1}\, M F M^{-1} = M A^r  F M^{-1}\]
 has the trace of $A^rF$.  Represent  $F$ as  a $2\times 2$ real matrix with  $(2,1)$-entry $c$; this entry is non-zero as $F$ is elliptic.   Thus, the trace of $PE$ equals $\text{tr} F + r (2 \cos \pi/m + 2 \cos\pi/n) c$.   However,   we can replace $P \in H$ by any power of $P$ and hence ensure that the trace of this {\em odd} element of $H$ is greater than 2 in absolute value.  That is, we have found an odd hyperbolic element in $H$,  obstructing this subgroup from being the Veech group of any translation surface. 
 \end{proof}

\section{Non-parabolic directions in the periodic field}  We show that 
certain infinite families of $G_{m,n}$ are such that any translation surface with Veech group commensurable to  $G_{m,n}$ has infinite classes of non-parabolic directions in its periodic field.   A search through these classes can reveal special affine pseudo-Anosov automorphisms,  as we found for the $(m,n) = (7,7)$ case; see Example~\ref{eg:sevenSeven}.   (This type of informed search was used in \cite{AS}.)

The technique we employ is number theoretic.  Since all entries of the generators given in~\eqref{e:generators}  are algebraic integers it easily follows that   $G_{m,n}$ is a  subgroup of $\text{SL}_2(\mathcal O_K)$,  where $\mathcal O_K$ is the ring of integers of the field $K= K_{m,n}$.   The quotient of $\mathcal O_K$ by any of its prime ideals is a finite field $\mathbb F$,   and there is induced homomorphism  from $G_{m,n}$ to a subgroup of $\text{SL}_2(\mathbb F)$, as well as from  the projectivisation of $G_{m,n}$ to $\text{PSL}_2(\mathbb F)$; these  {\em reduction homomorphisms} are defined by entry-wise reduction of the matrices in our group,  that is each entry is sent to its equivalence class modulo the ideal.  Now, $\text{SL}_2(\mathbb F)$ acts on the finite projective line $\mathbb P^1(\mathbb F)$,  and if the image of $G_{m,n}$ fails to act transitively there, then $G_{m,n}$ itself must fail to act transitively on $\mathbb P^1(K)$.    In this case,  the pre-images of elements not in the orbit of infinity are then  elements of $K$ that are not parabolic fixed points.     

 This method goes back to  Borho  \cite{B} and Borho-Rosenberger \cite{BR}, in the setting of the Hecke triangle groups, where it was further pursued by a school about Leutbecher,  see \cite{HMTY} for a recent usage in that setting. 
Underlying the method is the classification of the subgroups of the various $\text{PSL}_2(\mathbb F)$ by Dickson\cite{D} and Macbeath's \cite{M} application of this to study finite triangle groups.    We call the kernel of the reduction homomorphism a {\em congruence subgroup} of $G_{m,n}$.         Congruence subgroups of the Hecke triangle groups have been been studied for various reasons, see for example \cite{P}, \cite{LLT}.    The action of the full Galois group $\text{Gal}(\overline{ \mathbb Q}/\mathbb Q)$ on the algebraic curves uniformized by those congruence subgroups of Hecke groups corresponding to surjective reduction maps is studied in \cite{SS}. Indeed, as Macbeath showed,   the reduction of a triangle group (with parabolics) is almost always the full finite matrix group, see the recent work of Clark and Voight \cite{CV} for further discussion.  

\begin{Rmk}  There is a more elementary manner to prove the existence of non-parabolic points in special cases.    Indeed,  the trace field $K=K_{m,n}$ of $G_{m,n}$ is totally real,  thus in the setting where  $G_{m,n}$ is a  subgroup of $\text{SL}_2(\mathcal O_K)$,   the fact that the cusps of the Hilbert modular group  $\text{SL}_2(\mathcal O_K)$ are in 1-to-1 correspondence with the elements of the class group of $K$ shows that whenever  the class number of $K$ is greater than one there must be elements of $K$ that are non-parabolic fixed points.  See \cite{AS} for further discussion. 
\end{Rmk}

When considering quotients by prime ideals of rings of integers of fields generated by cosine values, the following lemma of Leutbecher is of great utility. 

\begin{Lem}\label{l:Leutbecher}[Leutbecher \cite{L}]  Given an integer $m\ge 3$,  let $\lambda = \lambda_m = 2 \cos \pi/m$.    If $m$ is not of the form twice a power of a prime, then $\lambda$ is a unit in the ring of integers $\mathcal O_K$ of $K = \mathbb Q(\lambda)$.   Otherwise,  if $m=2p^k$ for some prime $p$, $\lambda^{\phi(m)}$ is an associate in this ring of $p$;  here as usual $\phi(\cdot)$ denotes Euler's totient function. 
\end{Lem} 
       
Similarly,  we need the following.

\begin{Lem}\label{l:oddStandForm}  If at least one of  $m, n$ is odd, then $G_{m,n}$ is in parabolic standard form. 
 \end{Lem}

\begin{proof} Recall that both $0$ and $\infty$ are parabolic fixed points of any $G_{m,n}$.   
From \eqref{e:powersOfBandC}, 
we find that if   $m = 2 \ell +1$ or $n = 2 j + 1$ is odd, then  
\[ B^{\ell}\cdot (-1) = 0, \;\; \; C^{j}\cdot (1) = 0,\]
respectively.   This as,  $\sin \frac{(\ell +1) \pi}{2 \ell + 1} = \sin  \frac{\ell  \pi}{2 \ell + 1} $, and similarly in the other case.     Thus,  if at least one of $m,n$ is odd, then all three of $0, 1, \infty$ are parabolic fixed points of our group.  
\end{proof}

\subsection{Nonparabolic directions: $m=2^d$ case when odd $n \neq 2^f + 1$}\label{ss:intransM4}
\begin{Prop}\label{p:nonParabolics}    Suppose that  $m =  2^d$ with $d>1$ and that $n$ is odd with $n \neq 2^f + 1$ for any $f$.   Then $G_{m,n}$ is in parabolic standard form and is integrally normalized;  furthermore,  any finite index subgroup of $G_{m,n}$ that is realized as a Veech group is such that the corresponding translation surface has  non-parabolic directions with vanishing Sah-Arnoux-Fathi invariant. 
\end{Prop} 

\begin{proof}  By Lemma ~\ref{l:oddStandForm} $G_{m,n}$ is in parabolic standard form.  
By Lemma~\ref{l:Leutbecher}  with $m= 2^d$,  we find that the rational integral ideal $(2)$ factors as $(2 \cos \pi/m)^{\phi(m)}$.     
Now,  with  $K = K_{m,n}$  choose any prime ideal of  $\mathcal O_K$ lying above $(2 \cos \pi/m)$,  say $\mathfrak p$.    We have $\mathcal O_K/\mathfrak p \cong \mathbb F_{2^f}$, where $f$ is the residue degree of $\mathfrak p$.     This induces  a group homomorphism $\text{PSL}_2 (\mathcal O_K) \to \text{PSL}_2 ( \mathbb F_{2^f})$ that sends $B$ to an element of order two,  and hence the image of our group is a dihedral group of order $2n$.    Arguing as in \cite{BR},  this dihedral group is transitive on $\mathbb P^1(\mathbb F_{2^f})$ only if  $n = 2^f + 1$.     (Since $(2)$ is totally  ramified to $\mathbb Q(2 \cos \pi/m)$, the  residue degree of $\mathfrak p$ is the residue degree of the ideal of $\mathbb Q(2 \cos \pi/n)$ that $\mathfrak p$ lies above.)   

Thus, when $n$ is not of the form $2^f +1$   the orbit of infinity under $G_{m,n}$ does not equal all of   $\mathbb P^1(K)$.  That is,  there are elements of $\mathbb P^1(K)$ that are not parabolic fixed points. Since  $K =K_{m,n}$ also equals the invariant trace field $k_{m,n}$ by p. 159 \cite{MR},  and as verified in detail by Hooper, one has the  $K$ is the trace field of any finite index subgroup of $G_{m,n}$.     But,  the union of the parabolic fixed points of any such subgroup is simply the set of parabolic fixed points of $G_{m,n}$.   This is hence a proper subset of the trace field of the subgroup.    Since $G_{m,n}$ is in parabolic standard form,  so is any finite index subgroup; thus,  by Calta-Smillie, $\mathbb P^1(K)$ is the set of directions with vanishing Sah-Arnoux-Fathi invariant for the corresponding surface.
\end{proof}

\begin{Eg}  Let $m=4$ and $n=7$.  Recall that $\mathbb Z[2 \cos \pi/7]$ is the full ring of integers of $\mathbb Q(2 \cos \pi/7)$.   The minimal polynomial of $2 \cos \pi/7$ over $\mathbb Q$ (and hence over $\mathbb Z$, as this is an algebraic integer) is $p(x) = x^3 - x^2 - 2 x + 1$.    The reduction of  $p(x)$  modulo two  is irreducible; from this,  the ideal $(2)$ is inert to $\mathbb Q(2 \cos \pi/7)$,  and the quotient field $\mathbb Z[2 \cos \pi/7]$ modulo this ideal is thus a finite field of order $2^3$. 

 Indeed,  the orbit of $\infty$  modulo 2  is (by calculations, based on the fact that the orbit of $0$ is given by its orbit under just  the reduction of $B$,  and using the arithmetic of $\mathbb Q(2 \cos \pi/7)$ to simplify expressions)
\[ 0, \infty, \lambda, \lambda^2, 1, \lambda + 1, \lambda^2+\lambda\,\]
where $\lambda = 2 \cos \pi/7$.   Thus,  any element of $K$ in the $G_{4,7}$ orbit of any element of $\mathcal O_K$ equivalent to $\lambda^2+ \lambda + 1$ is {\em not} a parabolic fixed point.
\end{Eg}
\bigskip

 
 \subsection{Nonparabolic directions: $m=n$ case when odd $m$ not divisible by any  $2^f + 1$; another special pseudo-Anosov map}\label{ss:intransDiag}
\begin{Prop}\label{p:nonParabolicsDiag}    Suppose that  $m =  n$  is odd and not divisible by any integer $2^f + 1$ for positive $f$.   Any finite index subgroup of $G_{m,n}$ that is realized as a Veech group is such that the corresponding translation surface has  non-parabolic directions with vanishing Sah-Arnoux-Fathi invariant. 
\end{Prop} 

\begin{proof}  Again,  $G_{m,n}$ is in parabolic standard form.  The matrix $A$ is now clearly congruent to the identity modulo $2 \mathcal O_K$.   We thus choose a prime $\mathcal O_K$ ideal $\mathfrak p$ dividing this ideal.   By Lemma~\ref{l:Leutbecher}, neither $B$ nor $C$ is trivial when entries are reduced modulo $\mathfrak p$.    Thus,  $G_{m,m}$ projects to a non-trivial cyclic subgroup of $\text{SL}(\mathbb F_{2^f})$ where $f$ is the residue degree of $\mathfrak p$.   The order of this homorphic image must divide the orders of $B$ and $C$,  that is must divide $m$.   Since  $\mathbb P^1(\mathbb F_{2^f})$ has $2^f+1$ elements,   we conclude that this homomorphic image is too small to act transitively.     But then $G_{m,m}$ fails to act transitively on $\mathbb P^1(K)$.
\end{proof}  

\bigskip 
\begin{Eg}\label{eg:sevenSeven}  One again finds that the class of 
$\lambda^2+ \lambda + 1$ is not in the orbit of infinity.  In fact,  this element itself is fixed by 
a hyperbolic element for $G_{7,7}$;  for simplicity,  we take a conjugate element to get a simpler appearing matrix.   Let 
\[ M = A C^5 = \begin{pmatrix}-1 - 2 \lambda^2& -2 + 3 \lambda + 2 \lambda^2\\ -\lambda& -1 + \lambda^2\end{pmatrix}\]
(we have of course reduced entries modulo $p(x)$, the minimal polynomial of $\lambda$),  
so $M$ fixes 
\[ \dfrac{3 \lambda + \sqrt{ 7 \lambda^2 + \lambda -1}}{2 \lambda}\,.\]
 One calculates that $\beta =  (\alpha+13)/(\alpha-16)$ is a square root of $\alpha = 7 \lambda^2 + \lambda -1$.   Thus $M$ does have fixed points in the trace field.   There is correspondingly a special affine pseudo-Anosov homeomorphism of  totally real cubic dilatation $-9 \lambda^2 + 10 \lambda + 16$ on Hooper's $(Y_{7,7}, \eta_{7,7})$.   
 
\begin{figure}[h]%
\centering
\parbox{1in}{
\scalebox{.6}{\includegraphics{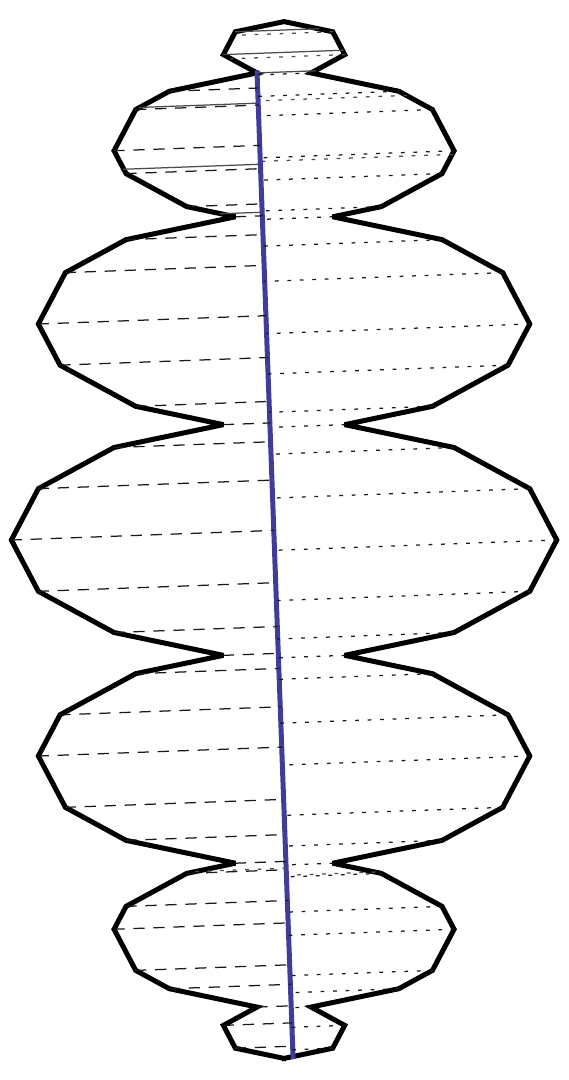}}
}%
\qquad \qquad \qquad 
\begin{minipage}{1in}%
\scalebox{.7}{\includegraphics{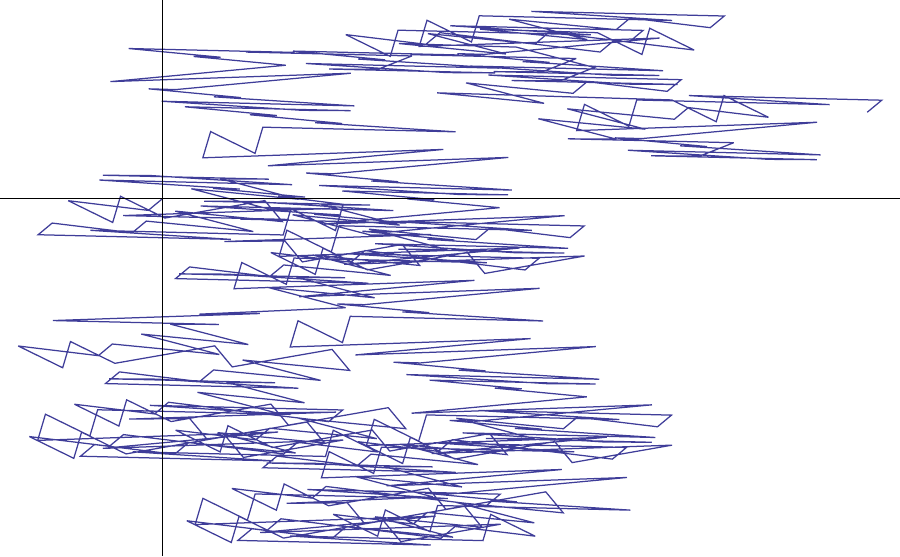}}
\end{minipage}%
\caption{A pseudo-Anosov map with vanishing Sah-Arnoux-Fathi invariant, indicated as zippered rectangles on the normalization of Hooper's translation surface $(Y_{7,7}, \eta_{7,7})$ such that all vertices have coordinates in the periodic field.  A deterministic walk approximating  \`a la McMullen \cite{Mc} a zero flux leaf of this rank three flow, as represented by using Galois automorphisms.}%
\label{f:fatSeven7andFluxMapped}%
\end{figure}
 
 The periodic field here is totally real and cubic (over the rationals), but
as $(Y_{7,7}, \eta_{7,7})$ is of genus 15,  this is certainly not the example of a  totally real cubic case of pseudo-Anosov map found by Lanneau and discussed by  McMullen in \cite{Mc}.    We now pursue McMullen's idea of focussing on the {\em rank} of the flow.   To normalize the surface so that all of the coordinates lie in the periodic field,   we divide all $x$-coordinates by $\sin \pi/7$ (including adjusting the flow direction, of course), see the left hand side of Figure~\ref{f:fatSeven7andFluxMapped}.      We choose a transversal (again in a direction perpendicular to that of our pseudo-Anosov with vanishing Sah-Arnoux-Fathi invariant) and explicitly find the interval exchange transformation given by first return to this transversal;  both the set of widths and of translations for this transformation are contained in the periodic field.   We can now consider an initial piece of the orbit under the interval exchange map of any point $x$ on the interval,   $(x_n)_{n\le N}$ and map this to $(\,(x_n-x)', (x_n-x)''\,)_{n\le N}$, the ordered pairs of the conjugates of the difference of the $n^{\text{th}}$ image from $x$, see the right hand side of Figure~\ref{f:fatSeven7andFluxMapped}. (There, after scaling the transversal interval to have length one,  we have taken $x = \lambda^{3}/8$.)   As MuMullen discusses,  this deterministic walk approximates the continuous leaf when identifying the first homology of the period torus of the real part of (the normalization of) $\eta_{7,7}$  with  the periodic field itself.
 
\end{Eg}


\begin{thebibliography}{Thurst88}   

\bibitem{A} P. Arnoux,
   {\em \'Echanges d'intervalles et flots sur les surfaces}, pp. 5--38 in
     \emph{Ergodic theory (Sem., Les Plans-sur-Bex, 1980)},
      Monograph. Enseign. Math., 29, Univ. Gen\`eve, Geneva, 1981. 

\bibitem{A2} --------,
  Th\`ese de 3$^{e}$ cycle, Universit\'e de Reims, 1981. 
  
  
  
\bibitem{ABB}P. Arnoux, J. Bernat, and X. Bressaud. {\em Geometrical models for
substitutions}, Exp. Math. 20 (2011), 97--127.  
  

\bibitem{AS} P. Arnoux and T. A. Schmidt, {\em Veech surfaces with non-periodic directions in the trace field},     J. Mod. Dyn.  3  (2009),  no. 4, 611--629.


\bibitem{AY} P. Arnoux and J.-C.  Yoccoz, {\em Construction de diff\'eomorphismes pseudo-Anosov},   C. R. Acad. Sci. Paris SŽr. I Math.  292  (1981), no. 1, 75--78.

 \bibitem{B}  W. Borho, {\em Kettenbr\"uche im Galoisfeld},  Abh. Math. Sem. Univ. Hamburg 39 (1973), 76--82. 
 
 \bibitem{BR}  W. Borho   and G. Rosenberger, {\em Eine Bemerkung zur Hecke-Gruppe $G(\lambda )$},  Abh. Math. Sem. Univ. Hamburg 39 (1973), 83--87.
 
\bibitem{BM}  I. Bouw and M. M{\"o}ller, {\em Teichm\"uller curves, triangle groups, and Lyapunov exponents}, Ann. of Math. (2) 172, (2010), 139--185.

\bibitem{BKS} R. M. Burton, C. Kraaikamp  and T. A. Schmidt,  
\emph{Natural extensions for the Rosen fractions}, Trans.\ Amer.\
Math.\ Soc.\ {\bf 352} (1999), 1277--1298.


\bibitem{C} K. Calta, {\em Veech surfaces and complete periodicity in genus 2},  J. Amer. Math. Soc.,  17  (2004),  no. 4, 871--908.

 
 
\bibitem{CaltaSchmidt} K. Calta and T. A. Schmidt, {\em  Continued fractions for a class of triangle groups},  to appear J. Austral. Math. Soc. 
 
 
\bibitem{CS} K. Calta and J. Smillie, {\em Algebraically periodic translation surfaces},   J. Mod. Dyn.  2  (2008),  no. 2, 209--248.


\bibitem{CV}   P. L. Clark and J. Voight, {\em Algebraic curves uniformized by congruence subgroups of triangle groups},  preprint (2011).


\bibitem{D}
L. E. Dickson, {\em Linear groups with an exposition of the Galois 
           field theory}  (republished), Dover, New York, 1958.

 


\bibitem{LLT}
M.-L. Lang, C.-H. Lim and S.-P. Tan, {\em Principal congruence subgroups of the Hecke groups},  J. Number Th.  85 (2000),220--230.  

	
 \bibitem{HMTY}
 E. Hanson, A. Merberg,  C. Towse,  and E. Yudovina,  {\em Generalized continued fractions and orbits under the action of Hecke triangle groups}  Acta Arith.  134  (2008),  no. 4, 337--348. 


\bibitem{HL} P. Hubert and E. Lanneau,  {\em  Veech groups without parabolic elements}, Duke Math. J. 133 (2006) , 335--346.
 
\bibitem{HS} P. Hubert and T. A. Schmidt,  {\em  Invariants of translation surfaces}, Ann. Inst. Fourier (Grenoble) 51 (2001), no. 2, 461--495.
 
\bibitem{LPV} J. H. Lowenstein, G. Poggiaspalla, and F. Vivaldi. {\em Interval exchange transformations over algebraic number fields: the cubic Arnoux-Yoccoz model}, Dyn. Syst. 22(2007), 73--106. 


\bibitem{LPV2} \bysame
{\em Geometric representation of interval exchange maps over algebraic number fields}, Nonlinearity 21 (2008), no. 1, 149--177.
 
\bibitem{H}  W. P. Hooper,  {\em Grid graphs and lattice surfaces},  preprint:
arXiv:0811.0799 (2009).

\bibitem{KS}
R. Kenyon and J. Smillie, {\em Billiards in rational-angled
triangles}, Comment. Mathem. Helv. \textbf{75}  (2000), 65--108.


\bibitem{L}   A. Leutbecher {\em \"{U}ber die
Heckeschen 
Gruppen G($\lambda$)}, {\em Abh. Math. Sem. Hamb.} 31 (1967), 199-205.



\bibitem{LR}  D. Long and A. Reid,  {\em Pseudomodular surfaces},   J. Reine Angew. Math.  552  (2002), 77--100.



\bibitem{M} A. M. Macbeath, {\em Generators of linear fractional groups},
in: Number Theory, Proc. Symp. in Pure Math. 12, 
W.J. Leveque and E.G. Straus, eds., 
 Amer. Math. Soc., Providence, 1969.




\bibitem{MR} C. Maclachlan and A. Reid, {\em The arithmetic of hyperbolic manifolds},  Springer,  GTM 219, 2003.

\bibitem{Mc} C. T. McMullen, {\em Cascades in the dynamics of measured
foliations},  preprint 2012. 

\bibitem{SS} T. A. Schmidt and K. M.  Smith, {\em Galois orbits of principal congruence Hecke curves},  J. London Math. Soc. (2) 67 (2003), no. 3, 673--685.

\bibitem{P} L. A. Parson, {\em  Normal congruence subgroups of the Hecke groups $G(2^{1/2})$ and $G(3^{1/2})$}, 
Pacific J. Math. 70 (1977), no. 2, 481--487. 

\bibitem{T} W. Thurston, 
{\em On the geometry and dynamics of diffeomorphisms of surfaces},  
Bull. A.M.S. 19 (1988), 417 -- 431.


\bibitem{V} W. A. Veech, {\em Teichm\"uller curves in modular
space, Eisenstein series, and an application to triangular billiards},
Inv. Math. 97 (1989), 553 -- 583.

\bibitem{W}  C. Ward  {\em Calculation of Fuchsian groups associated to billiards in a rational triangle}, Ergodic Theory Dynam. Systems 18,  (1998), 1019--1042. 
\end{thebibliography}
\end{document}